\documentclass{amsart}
\usepackage{amssymb, amsmath, amscd}
\usepackage[all]{xy}
\swapnumbers

\renewcommand{\a}{\alpha}
\renewcommand{\b}{\beta}

\renewcommand{\d}{\delta}

\providecommand{\s}{\sigma}

\providecommand{\om}{\omega}
\providecommand{\Om}{\Omega}

\newcommand{\R}{\mathbb R}

\newcommand{\Z}{\mathbb Z}

\newcommand{\calA}{\mathcal{A}}
\newcommand{\calB}{\mathcal{B}}
\newcommand{\calC}{\mathcal{C}}
\newcommand{\calG}{\mathcal{G}}
\newcommand{\calF}{\mathcal{F}}
\newcommand{\calH}{\mathcal{H}}

\newcommand{\calN}{\mathcal{N}}

\newcommand{\calU}{\mathcal{U}}

\newcommand{\calZ}{\mathcal{Z}}

\newcommand{\conj}{\overline}

\newcommand{\order}{\operatorname{or}}

\newcommand{\coker}{\operatorname{coker}}
\newcommand{\im}{\operatorname{im}}

\newcommand{\End}{\operatorname{End}}
\newcommand{\Pf}{\operatorname{Pf}}

\newcommand{\lk}{\operatorname{lk}}
\newcommand{\IE}{\operatorname{(I/E)}}

\numberwithin{equation}{section}

\theoremstyle{plain}
\newtheorem{thm}[equation]{Theorem}

\newtheorem{cor}[equation]{Corollary}
\newtheorem{prop}[equation]{Proposition}

\theoremstyle{remark}
\newtheorem{remark}[equation]{Remark}
\newtheorem{ex}[equation]{Example}

\theoremstyle{definition}
\newtheorem{defn}[equation]{Definition}

\begin{document}

\title{Subdivisions and Transgressive Chains}
\author{Jer-Chin (Luke) Chuang}
\thanks{The author thanks Robin Forman for many helpful discussions.} 

\begin{abstract}
Combinatorial transgressions are secondary invariants of a space admitting triangulations.  They arise from subdivisions and are analogous to transgressive forms such as those arising in Chern-Weil theory.  Unlike combinatorial characteristic classes, combinatorial transgressions have not been previously studied.  First, this article characterizes transgressions that are path-independent of subdivision sequence.  The result is obtained by using a cohomology on posets that is shown to be equivalent to higher derived functors of the inverse (or projective) limit over the opposite poset.  Second, a canonical local formula is demonstrated for a particular combinatorial transgression: namely, that relative the difference of Poincar\'{e} duals to the Euler class.
\end{abstract}

\maketitle

\section{Introduction}

In recent years there has been renewed interest in combinatorial methods for differential geometry and topology as motivated by computational considerations.  Central is the concept of a triangulation which provides a combinatorial representation amenable to computational methods.  Applications of triangulations are diverse, ranging from computer-aided design and mesh-generation for numerical methods to geographic information systems, computer graphics, and the visualization and geometrization of large data sets (see for example the books \cite{hjelle}, \cite{edelsbrunner}, \cite{zomorodian} and references therein).  A typical scenario is as follows: one samples a complicated object, such as geographic terrain, and hopes to reconstruct some computationally amenable model from the resulting data set, often a triangulation.  If particular regions of the terrain are further sampled, the resulting additional data can be realized (under appropriate triangulation methods) as a subdivision of the original model.  

The archetype for relationships between geometry and topology is the Gauss-Bonnet theorem for surfaces.  For triangulated surfaces, the combinatorial analogue replaces the integral of Gaussian curvature over the surface with a sum over the vertices of a likewise local quantity, the \emph{angle defect}.  The angle defect at a vertex is $2\pi$ minus the sum of the interior angles at the vertex.  This notion accords with the usual intuitive picture of flat, positively, or negatively-curved surfaces.  In the 1940s, Chern\cite{chern1} reinterpreted the Gauss-Bonnet theorem as the computation of the characteristic number associated to the Euler class, a characteristic class.  Characteristic classes are cohomology classes in the base space of a vector bundle and are helpful in distinguishing isomorphism classes of vector bundles.  One may then ask for combinatorial formulae for the homology duals of characteristic classes.  The existence of local formulae was shown by Levitt and Rourke\cite{levitt}, and many have pursued this still open question: Gabrielov, Gelfand, and Losik\cite{gabrielov}, Cheeger\cite{cheeger}, Gelfand and MacPherson\cite{gelfand}, and more recently, Gaifullin\cite{gaifullin}.

One particular approach to characteristic classes, namely Chern-Weil theory, suggests that when primary characteristic classes coincide, there is a finer invariant called a \emph{transgression}.  Though combinatorial formulae for primary characteristic classes have been much studied, combinatorial transgressions have not been previously examined and are the subject of this paper.

We briefly recall the Chern-Weil theory (see Section \ref{diffgeo} for details): Given a vector bundle over a smooth manifold, Chern-Weil theory provides a local expression for forms representing characteristic classes of the bundle.  The construction is effected via a map from the invariant polynomials on the general linear group to the cohomology ring of the base.  This method involves the choice of a connection, but for a fixed invariant polynomial, the local expressions obtained from two connections differ by an exact form.  A form whose exterior derivative is such a difference is said to be \emph{transgressive}.  Chern-Weil theory also provides a formula for a natural choice of transgressive form relative two given connections.  This is effected by a canonical deformation in the affine space of connections.  A particular case was studied by Chern and Simons\cite{chern2} and is now called a \emph{Chern-Simons form}.  It is a secondary invariant dependent on both the bundle and a choice of connection and in particular cases defines cohomology classes called \emph{secondary characteristic classes}.  These forms are important in the study of flat bundles and also in theoretical physics.

In the combinatorial setting, the transgression is a chain obtained from a simplicial complex and a subdivision of it.  The transgression provides a subdivision-dependent secondary invariant for spaces admitting triangulations and leads to interesting questions regarding the poset of triangulations and its homology with respect to certain systems of vector spaces.  This paper characterizes cases when the transgression is independent of subdivision sequences between the two complexes and shows that there is a canonical local formula for a particular transgression, namely that relative to certain $0$-cycles called \emph{Euler cycles}.  Both of these properties were motivated by results from the Chern-Weil context.  (See Appendix for details.)

Formally, the combinatorial transgression may be defined as follows: Let $M$ be a space admitting triangulations, and fix a homology class $h\in H_k(M)$.  For each triangulation we may choose a representative for the given homology class.  Suppose $Y$ is a triangulation and $i\colon Y\to X$ a subdivision with representative cycles $\b,\a$ on $C_k(Y),C_k(X)$ respectively.  Then, the chain $\a-i_*\b$ is a boundary, and any $(k+1)$-chain whose boundary is such is called a \emph{transgression} relative to the representatives $\a,\b$ under the subdivision $i\colon Y\to X$.  Note that transgressions are defined relative to a manner of choosing representatives and are determined up to cycles.

Transgressions are thus secondary invariants that reflect the difference between complexes related by subdivision.  We say a transgression is \emph{path-independent} provided for a sequence of subdivisions $X_1\stackrel{i}{\to} X_2\stackrel{j}{\to} X_3$ the transgression associated to $i,j$ add to yield that for $j\circ i$.  In particular, path-independent transgressions may be computed step-wise with respect to any subdivision sequence connecting the initial and final complexes.  If we view transgressions as chain-valued $1$-cochains on the order complex of the poset of triangulations ordered by subdivision, then path-independence may be rephrased as requiring the transgression actually to be a $1$-cocycle.  Section \ref{order} develops a cohomology (called \emph{order cohomology}) on posets that formalizes on this observation.  This cohomology is shown to be equivalent to the derived functors of the inverse (or projective) limit with respect to the opposite underlying poset.  (For more on higher derived functors of the inverse limit, see Roos\cite{roos}, Jensen\cite{jensen}, or Weibel\cite{weibel}.)  We will show near the end of this paper that for sufficiently nice posets like lattices, order cohomology is isomorphic to a particular sheaf cohomology.  Regardless, using order cohomology, we obtain the following:

\begin{thm}\label{pathindependence}
Consider a poset of triangulations for a fixed space directed by subdivision.  Suppose there is no simplicial cycle common to every triangulation in the poset.  Then, for any cycle-representative assignment, there is a compatible path-independent transgression iff all the cycle-representatives are actually boundaries, that is, the trivial homology element is being represented.
\end{thm}

For example, suppose for each triangulation the chain spaces are endowed with an inner product.  Then, Hodge theory provides unique cycle representatives for each homology class.  For spaces satisfying the hypothesis in the above theorem, transgressions relative to these harmonic cycles are not path-independent whenever a non-zero harmonic cycle is assigned.  An example where there is a simplicial cycle common to every triangulation in the poset is the subdivision poset of a fixed simplicial complex.  However, in this case we can always find a path-independent transgression:

\begin{cor}
For a fixed triangulation, the poset of its subdivisions admits a path-independent transgression (possibly) trivial.  Using the canonical inner product on chain spaces of the subdivisions yields an unique path-independent transgression.
\end{cor}

As mentioned previously, Chern-Weil theory not only provides local formulae for characteristic forms, but also for a canonical transgression (see the Appendix).  The second main result is that there is an analogous canonical combinatorial transgression relative certain $0$-cycles called \emph{Euler cycles}.  On combinatorial manifolds these cycles are homology duals to the Euler class.  The notion of locality used is that of a \emph{party}:  Let $i\colon N\to M$ be a subdivision of simplicial complexes.  A \emph{$k$-party} of $M$ is a union of $k$-simplices of $M$ which coincides with the image of a $k$-simplex of $N$ under the map $i$.
\begin{thm}\label{localformula}
There is a canonical formula (with rational coefficients) for the transgression relative to the Euler cycles.  This formula is local in the sense that the coefficient for any edge in the transgressive chain is determined by the union of top-dimensional parties containing that edge.
\end{thm}

The sections are organized as follows: Section \ref{order} introduces a convenient homological setting for examining transgressions, called \emph{order cohomology}, and explains its relation to derived functors.  Section \ref{pathind} recasts the combinatorial transgression in this setting to deduce the above results on path-independent transgressions, and Section \ref{local} demonstrates a local formula for a transgression relative Euler cycles.  Finally, Section \ref{sheaf} explores the connections between order cohomology and sheaf theory.  An appendix provides the motivations from differential geometry.

\begin{remark}
The results of this article were part of the author's Ph.D. dissertation.  However, the connections between order cohomology and derived functors were only subsequently realized.  The exposition via poset cohomology has been retained here for concreteness.
\end{remark}

\section{Order Cohomology}\label{order}

This section introduces a cohomology on posets called \emph{order cohomology} that is well-suited for studying transgressions and questions of path-independence.  Subsection \ref{furtherex} provides a couple of longer examples of computing order cohomology while Subsection \ref{derivedfunctors} explains the relationship between order cohomology and the derived functors of the categorical limit.

Let $P$ be a poset.  By a system $\calG$ of abelian groups directed by poset $P$ we mean a direct system of abelian groups where the directed set $P$ need not have upper bounds.  We will denote the collection of abelian groups by  $\{\calG_p\}_{p\in P}$ and the morphisms by $\{\Phi(p,q)\}_{p\preceq q}$.  In other words, $\calG$ is a covariant functor from the poset category $P$ to the category of abelian groups.  Denote the order complex of $P$ by $X=\order(P)$, and let $\calG_p$ be the group indexed at $p\in P$.

\begin{defn}
A \emph{$k$-cochain with values in $\calG$} is an assignment of a group element $\varphi(p_0\cdots p_k)\in G_{p_k}$ for each $k$-simplex $p_0\cdots p_k$ of the order complex $X$.
\end{defn}

The set of such will be denoted $C^k(X;\calG)$ and is naturally an abelian group isomorphic to $\prod \calG_{\om(\s)}$ where $\s$ ranges over $k$-simplices of $X$ and $\om(\s)$ is the top vertex of $\s$.

For any subset $Y\subset X$ of the order complex (as an abstract simplicial complex), let $C^*(Y;\calG)$ denote those with support in $Y$.  There is a natural coboundary map $\d^k\colon C^k(Y;\calG)\to C^{k+1}(Y;\calG)$ given in C\v{e}ch-like fashion:
\begin{equation}
(\d^k \varphi)(p_0\cdots p_{k+1}) = \sum_{i=0}^k (-1)^i \varphi(p_0\cdots \hat{p_i}\cdots p_k) +\Phi(p_k,p_{k+1})\varphi(p_0\cdots p_k)
\end{equation}
where for any pair $p\preceq q$, $\Phi(p,q)\colon \calG_p\to \calG_q$ is the directed homomorphism from $p$ to $q$ in the system.  One checks that $\d\circ\d=0$.

\begin{defn}
The cohomology of the complex $\{C^*(Y;\calG),\d\}$ will be called the \emph{order cohomology} of the subset $Y$ relative the system $\calG$.  It will be denoted $H^*(Y;\calG)$.
\end{defn}

\begin{ex}
Let $P$ be a poset and $\calG$ a system of abelian groups such that $\calG_p\cong G$ for some fixed abelian group $G$.  Then, $H^k(X;\calG)$ is naturally isomorphic to the simplicial cohomology of $X=\order(P)$ with coefficients in $G$.
\end{ex}

\begin{ex}[Wedge]\label{wedge}
Consider the poset $P$ on three elements given by $a,b\preceq c$.  Let $\calG$ be a system directed by $P$ with vector spaces $\{V_p\}$ and morphisms $\{\varphi_{pq}\}$.  We show that the order cohomology $H^*(\order(P);\calG)$ may be non-trivial.

By definition $0$-cocycles are assignments $v_p\in V_p$ such that $\varphi_{pc}(v_p)=v_c$ for $p=a,b$ so that the non-triviality of the $0$-th order cohomology is dependent on the maps $\varphi_{pc}$.  For example, if the morphisms $\varphi_{ac},\varphi_{bc}$ share non-trivial image, then $H^0\neq 0$.  Also, by definition the first order cohomology is $(V_c\oplus V_c)/\langle (v_c-\varphi_{ac}(v_a),v_c-\varphi_{bc}(v_b))| v_a\in V_a, v_b\in V_b, v_c\in V_v\rangle$.  In particular, if $V_a=V_b=0$ and $V_c\neq 0$, then $H^1\neq 0$.

Note that the set $\{a,b\}$ determines a cross-cut in $P$.  Since the above analysis shows that $H^1$ need not be trivial, we see that a Cross-Cut Theorem does not hold for order cohomology.
\end{ex}

For simplicity, we henceforth write $C^*(Y)=C^*(Y;\calG)$ when the system is clear from context.  There is a natural notion of relative chain groups $C^*(X,Y;\calG)$ satisfying the short-exact sequence:
\begin{equation}
0\to C^*(Y)\to C^*(X)\to C^*(X,Y)\to 0
\end{equation}
with associated long-exact sequence on homology:
\begin{equation}
\cdots\to H^k(Y)\to H^k(X)\to H^k(X,Y)\to H^{k+1}(Y)\to\cdots
\end{equation}
Since $X=\order(P)$ is a simplicial complex, there is an analogous Mayer-Vietoris sequence on order cohomology for any covering of the complex.

However, there is no general Mayer-Vietoris sequence for covers relative the poset $P$.  Let $\{\Delta_1,\Delta_2\}$ be a covering for $P$.  In general, the associated order complexes $U_i=\order(\Delta_i)$ do not cover the order complex $X=\order(P)$. Let $Y$ be the set difference (as abstract simplicial complexes) and note that $C^*(U_1\cap U_2)=C^*(\order(\Delta_1\cap\Delta_2))$.  The sequence
\begin{equation}
0\to C^*(X)\to C^*(U_1)\oplus C^*(U_2)\to C^*(U_1\cap U_2)\to 0
\end{equation}
given by restriction and difference, respectively, fails to be exact only at $C^*(X)$ because the kernel of the restriction map is the set of cochains with support not subordinate to the coverings $U_i$, namely $C^*(Y)$.  Thus, the sequence
\begin{equation}\label{pseudoMV}
0\to C^*(Y)\to C^*(X)\to C^*(U_1)\oplus C^*(U_2)\to C^*(U_1\cap U_2)\to 0
\end{equation}
is exact.  Because the simplices of $X$ form a canonical basis for $C^*(X)$, the elements of $Y$ are a subbasis and induce a canonical splitting $Q^*(X,Y)\subset C^*(X)$:
\begin{equation}
0\to C^*(Y) \to C^*(X)\leftrightharpoons Q^*(X,Y)\to 0
\end{equation}
where $Q^*(X,Y)\cong C^*(X,Y)$.  Precisely, $Q^*(X,Y)$ consists of cochains supported on simplices subordinate to the cover.  Thus, we can splice Equation (\ref{pseudoMV}) to obtain a short-exact sequence:
\begin{equation}
0\to Q^*(X,Y)\to C^*(U_1)\oplus C^*(U_2)\to C^*(U_1\cap U_2)\to 0
\end{equation}
with associated long-exact sequence:
\begin{equation}
\cdots\to H^k(Q)\to H^k(U_1)\oplus H^k(U_2)\to H^k(U_1\cap U_2)\to H^{k+1}(Q)\to\cdots
\end{equation}
where $H^*(Q)$ is the order cohomology relative the chain complex $Q^*(X,Y)$.

\begin{ex}
When the cover $\{U_i\}$ is disjoint, $H^*(U_1\cap U_2)$ vanishes, so that $H^*(Q)\cong H^*(U_1)\oplus H^*(U_2)$.  If furthermore, the $\Delta_i$ are incomparable, then the set difference $Y=0$ so that
\begin{equation}
H^*(X)\cong H^*(Q)\cong H^*(U_1)\oplus H^*(U_2)
\end{equation}
as expected.
\end{ex}

\begin{ex}
If $f\colon P\to Q$ is an order-preserving map of posets, and $\calG$ a system directed by $Q$, there is an induced pull-back system $f^*\calG$ given by setting $(f^*\calG)_p=\calG_{f(p)}$ with maps $\varphi_{pp'}=\varphi_{f(p)f(p')}$ for $p,p'\in P$.
\end{ex}

\begin{remark}
A push-forward system is not well-defined.  For example, let $P=\{a,b\preceq c\preceq d\}$ and $Q=\{x\preceq y\preceq z\}$ and consider the order-preserving map $f(a)=f(c)=y$ and $f(b)=x$ and $f(d)=z$.
\end{remark}

\subsection{Further Examples}\label{furtherex}
In this subsection we present two longer examples of order cohomology.  The first will be needed later in the proofs of Corollary \ref{fixed} and Proposition \ref{coincide}.  The second illustrates the computation of order cohomology via a covering of the directing poset.  As explained earlier, care needs to be taken to account for cochains with support not subordinate to the covering.

\begin{prop}[Inverted cone]\label{acyclic}
Let $P$ be a poset with an unique minimal element $p\in P$, and $\calG$ a system of abelian groups directed by $P$.  Then, the order cohomology of $X=\order(P)$ is
\begin{equation}
H^k(X)=H^k(X;\calG)\cong
\begin{cases}
\calG_p\quad k=0\\
0\quad k>0
\end{cases}
\end{equation}
so that $X$ is acyclic relative any system $\calG$.
\end{prop}
\begin{proof}
First, we show that $H^0(X)\cong \calG_p$.  For any $g\in \calG_p$, we may define a cocycle $[\gamma]\in H^0(X)$ by setting $\gamma(p)=g$ and $\gamma(q)=\Phi(p,q)g$ for any $q\succeq p$.  Evidently, any element of $H^0(X)$ arises in this manner, and the construction defines an isomorphism.

Now, we show that $H^k(X)$ vanishes for $k>0$.  Let $\gamma\in \ker \d^k$ be a cocycle.  We claim that the following conditions define a cochain $\b\in C^{k-1}(X;\calG)$ such that $\d\b=\gamma$:
\begin{enumerate}
\item Set $\b$ arbitrarily on $(k-1)$-simplices in $X$ containing $p=\min(P)$.
\item For any $(k-1)$-simplex $q_1\cdots q_k$ of $X$, define $\b(q_1\cdots q_k)$ by
\begin{multline}
\gamma(pq_1\cdots q_k)=\b(q_1\cdots q_k)\\
+\sum_{i=1}^{k-1} (-1)^i \b(pq_1\cdots \hat{q_i}\cdots q_k)+(-1)^k\Phi(q_{k-1},q_k)\b(pq_1\cdots q_{k-1})
\end{multline}
\end{enumerate}
By construction, $\gamma=\d\b$ for all $k$-simplices of $X$ containing $p$.  Now, note that for any $k$-simplex $q_1\cdots q_{k+1}$,
\begin{multline}
0 = (\d^2\b)(pq_1\cdots q_{k+1}) = (\d\b)(q_1\cdots q_{k+1})\\
+\sum_{i=1}^k (-1)^i (\d\b)(pq_1\cdots \hat{q_i}\cdots q_{k+1})+(-1)^{k+1}\Phi(q_k,q_{k+1})(\d\b)(pq_1\cdots q_k)
\end{multline}
Because $\gamma$ is a cocycle,
\begin{multline}
0 = (\d\gamma)(pq_1\cdots q_{k+1}) = \gamma(q_1\cdots q_{k+1})\\
+\sum_{i=1}^k (-1)^i \gamma(pq_1\cdots \hat{q_i}\cdots q_{k+1})+(-1)^{k+1}\Phi(q_k,q_{k+1})\gamma(pq_1\cdots q_k)
\end{multline}
Comparing the two equations, we conclude that $\gamma=\d\b$ for arbitrary $k$-simplices.
\end{proof}

\begin{prop}[Double cone]
Let $P$ be a poset that is the union of two subposets $\Delta_1,\Delta_2$ which intersect at precisely one point $p\in P$ which is simultaneously the unique minimal element of $\Delta_1$ and the unique maximal element of $\Delta_2$.  Let $\calG$ be a system of abelian groups directed by $P$, and write $X=\order(P)$ and $U_i=\order(\Delta_i)$.  Then,
\begin{equation}
H^1(X;\calG)\cong H^1(U_2;\calG)
\end{equation}
and if $\calG$ were furthermore a system of vector spaces,
\begin{equation}
H^0(X;\calG)\cong H^0(U_2;\calG)
\end{equation}
\end{prop}
\begin{proof}
Since the system $\calG$ is fixed, we will write $H^*(X)$ for $H^*(X;\calG)$.  Let $Y=X-U_1-U_2$.  Note that we have the diagram with exact rows,
\begin{equation*}
\begin{CD}
0 @>>> C^1(Y) @>>> C^1(X) @>>> \oplus_i C^1(U_i) @>>> C^1(U_1\cap U_2) @>>> 0\\
& & @AAA @AA\d A @AA\d A @AA\d A\\
& & 0 @>>> C^0(X) @>\rho>> \oplus_i C^0(U_i) @>>> C^0(U_1\cap U_2) @>>> 0 
\end{CD}
\end{equation*}
given by restriction and difference.  The fact that $C^1(X)$ does not inject for arbitrary $\calG$ means that the connecting homomorphism on homology is not uniquely defined.  We will exploit this non-uniqueness.

Let $\gamma\in C^0(U_1\cap U_2)\cong \calG_p$ be a cocycle.  It lifts to an element $(\a,\b)\in C^0(U_1)\oplus C^0(U_2)$ and we may assume that $\a(p)=0$ by adjusting with an element of $\im \rho$.  Then, using the value $\b(p)$ we may extend $\b$ over $U_1$, and this extension $\conj{\b}$ agrees with $\gamma$ at $p$.  Define $\tau\in C^1(X)$ by
\begin{equation*}
\tau(a,b)=
\begin{cases}
\d\conj{\b}(ab) \quad \text{if}\quad a,b \quad \text{are vertices in the same cover element} \quad U_i\\
\d\conj{\b}(ap) + \d\conj{\b}(pb) \quad a\preceq p\preceq b
\end{cases}
\end{equation*}
It is a coboundary in $C^1(X)$, and we may define a connecting homomorphism by $d[\gamma]=[\tau]$.  Then, $\im d=0$ and the long-exact sequence,
\begin{equation*}
0\to H^0(X)\to \oplus H^0(U_i)\to H^0(U_1\cap U_2)\stackrel{d}{\to} H^1(X)\to \oplus H^1(U_i)\to 0
\end{equation*}
implies that
\begin{equation}
H^1(X)\cong H^1(U_1)\oplus H^1(U_2) = H^1(U_2)
\end{equation}
by the preceding example since $\Delta_1$ has an unique minimal element.  Also, the resulting short exact sequence
\begin{equation*}
0\to H^0(X)\to H^0(U_1) \oplus H^0(U_2)\to H^0(U_1\cap U_2)\to 0
\end{equation*}
splits for vector spaces to yield the second assertion.
\end{proof}

\subsection{Various Homological Constructions}\label{derivedfunctors}
Let $P$ be a poset and recall that a system of abelian groups is a covariant functor from the poset category $P$ to abelian groups.  In this subsection we will show that if $P^{op}$ denotes the opposite category of $P$ and $X^{op}$ the associated order complex, then under a technical assumption, $H^k(X^{op};\calG)\cong \lim^k \calG$ the derived functors of the categorical limit.  Except for Equations (\ref{sescomplex}) and (\ref{leshomology}), most of the results in this subsection are not needed explicitly in subsequent sections.

We start by showing that $H^0(X^{op};\calG)=\lim \calG$.  Let $\pi_p\colon \prod_i \calG_i \to \calG_p$ the projection onto the component $\calG_p$ indexed by $p\in P$, and recall that $\calG$ is a covariant functor.  Then
\begin{equation}
H^0(X^{op};\calG) = \left\{g\in \prod_i \calG_i| \pi_p(g)=(\Phi_{pq}\circ\pi_q)(g), p\preceq q \right\} \cong \lim \calG
\end{equation}
Hence, $H^0(X;\calG)$ is the categorical limit of $\calG$ with respect to the opposite underlying poset category $P^{op}$.

Let $\calG, \calF$ be systems of abelian groups over a poset $P$ with $X=\order(P)$.  Suppose $\psi\colon \calG\to \calF$ is a morphism of systems.  There is an induced map $\psi_*\colon C^k(X;\calG)\to C^k(X;\calF)$ given by push-forward $\psi_*(\varphi)=\psi\circ \varphi$.  Note that
\begin{equation*}
\psi_*\d\varphi = \psi_*\sum (-1)^i \varphi[i] = \sum (-1)^i \psi_*\varphi[i] = \sum (-1)^i (\psi_*\varphi)[i] = \d\psi_*\varphi
\end{equation*}
where for any $k$-simplex, $[i](v_0\cdots v_k)=v_0\cdots\hat{v_i}\cdots v_k$ the omission of the $i$-th vertex, and we have notationally suppressed $\Phi$.  Hence $\d,\psi_*$ commute, so that the cochain maps descend to cohomology $\psi_*\colon H^k(X;\calG)\to H^k(X;\calF)$, and this defines a covariant functor $H^k(X;-)$ from systems over $P$ to abelian groups.  In particular, if $\calG,\calF$ are systems of chain complexes with $\psi\colon \calG\to \calF$, then there are induced boundary maps
\begin{align*}
\partial_k^{\calG}\colon& C^k(X;\calG_j)\to C^k(X;\calG_{j-1})\\
\partial_k^{\calF}\colon& C^k(X;\calF_j)\to C^k(X;\calF_{j-1})
\end{align*}
The cochain map $\varphi_*$ preserves bigrading and as before descends to a homology map $\psi_*\colon H^k(X;\calG_j)\to H^k(X;\calF_j)$.

Next, let $\psi\colon \calA\to \calB$ be a morphism of systems over poset $P$, and let $p\preceq q$ in $P$.  The diagram
\begin{equation*}
\begin{CD}
\calA_q @>\psi_q>> \calB_q\\
@A\Phi_{pq}^{\calA}AA @AA\Phi_{pq}^{\calB}A\\
\calA_p @>\psi_p>> \calB_p
\end{CD}
\end{equation*}
shows that $\Phi_{pq}(\ker \psi_p)\subset \ker \psi_q$ and $\Phi_{pq}(\im \psi_p)\subset \im \psi_q$, so that the systems $\ker \psi$, $\im \psi$, and $\coker \psi$ are well-defined, and we may define short-exact sequences of systems.  Using the fact that for any system $\calG$, the cochain group $C^k(X;\calG)$ is naturally isomorphic to $\prod \calG_{\om(\s)}$ where $\s$ ranges over $k$-simplices of $X$ and $\om(\s)$ is the top vertex of $\s$, one readily checks that $C^k(X;-)$ is an exact covariant functor from systems over $P$ to abelian groups.

For a short-exact sequence of systems, $0\to \calA\to \calB\to \calC\to 0$, because the diagram
\begin{equation}
\begin{CD}
C^k(X;\calA) @>>> C^k(X;\calB) @>>> C^k(X;\calC)\\
@V\d VV @V\d VV @V\d VV\\
C^{k+1}(X;\calA) @>>> C^{k+1}(X;\calB) @>>> C^{k+1}(X;\calC)
\end{CD}
\end{equation}
commutes, we get a short-exact sequence of complexes:
\begin{equation}\label{sescomplex}
0\to C^k(X;\calA) \to C^k(X;\calB) \to C^k(X;\calC)\to 0
\end{equation}
and a long-exact sequence on order cohomology:
\begin{equation}\label{leshomology}
\cdots\to H^k(X;\calA)\to H^k(X;\calB)\to H^k(X;\calC)\stackrel{d}{\to} H^{k+1}(X;\calA)\to\cdots
\end{equation}
For example, the order cohomology of the system $\calF\oplus\calG$ fits into the long-exact sequence:
\begin{equation*}
\cdots\to H^k(X;\calF)\to H^k(X;\calF\oplus\calG)\to H^k(X;\calG)\to H^{k+1}(X;\calF)\to\cdots
\end{equation*}
Such long-exact sequences will be used in proving the main result in Section \ref{pathind}.  Also, for a morphism of short-exact sequences of systems:
\begin{equation}
\begin{CD}
0 @>>> \calA @>>> \calB @>>> \calC @>>> 0\\
& & @V\a VV @V\b VV @V\gamma VV\\
0 @>>> \calA' @>>> \calB' @>>> \calC' @>>> 0
\end{CD}
\end{equation}
a standard ``diagram chase'' shows that the connecting homomorphism in the long-exact sequence on homology is natural:
\begin{equation}
\begin{CD}
H^k(X;\calC) @>d >> H^k(X;\calA)\\
@V\gamma_* VV @V\a_* VV\\
H^k(X;\calC') @>d >> H^k(X;\calA')
\end{CD}
\end{equation}

\begin{remark}
The preceding results on exact sequences show that $H^k(X;-)$ is a cohomological $\d$-functor.
\end{remark}

Next, we show that the functors $H^k(X;-)$ are left-exact:  Let $\colon \calA\stackrel{f}{\to} \calB\to \calC$ be exact with induced sequence $0\to H^k(X;\calA)\to H^k(X;\calB)\to H^k(X;\calC)$.  Exactness at $H^k(X;\calB)$ is routine.  For exactness at $H^k(X;\calA)$, note that $f_*$ is injective as a cochain map.  Let $\s=\d\tau$ where $\tau\in C^{k-1}(X;\calB)$.  We show that the pre-image of $\s$ is a coboundary:  By injectivity of $f_*$ there is an unique inverse cochain $\rho=f_*^{-1}\tau\in C^{k-1}(X;\calA)$ define by $\rho(v_0\cdots v_{k-1}) = f_*^{-1}\tau(v_0\cdots v_{k-1})$.  Then, $f_*\d\rho=\d f_*\rho = \d\tau=\s$ so that $\d\rho$ is the unique inverse of $f_*^{-1}\s$, and the homology map $f_*\colon H^k(X;\calA)\to H^k(X;\calB)$ is injective.

Finally, suppose $\calG$ possesses a resolution
\begin{equation*}
0\to \calG\to \calG^{(0)}\to \calG^{(1)}\to \cdots
\end{equation*}
by order-acyclic systems $\calG^{(k)}$, then order cohomology is an \emph{universal} cohomological $\d$-functor, and by standard uniqueness results (see Weibel\cite{weibel}), we have
\begin{equation*}
H^k(X^{op};\calG) = {\lim}^k \calG
\end{equation*}
the derived functors of the categorical limit.  However, it is unclear to the author whether order-acyclic resolutions exist for general systems.

\section{Path-Independent Transgressions}\label{pathind}

In this section, we prove the main result (Theorem \ref{path-ind}) on path-independent transgressions by viewing transgressions as chain-valued 1-cocycles on the order complex of the poset of triangulations under subdivision.  The statement of the result involves a mild technical assumption.  Since the assumption does not hold for the important case of subdivisions of a fixed simplicial complex, we treat this case separately in Corollary \ref{fixed}.

Let $M$ be a space admitting triangulations.  Recall that a triangulation of $M$ is a simplicial complex with a homeomorphism from its geometric realization to $M$.  Let $P$ be the poset of triangulations of $M$ ordered by subdivision, and define a system $\calC$ of vector spaces directed by $P$ where for each triangulation $p\in P$, we have the graded chain space $C_*(p)$ of the triangulation $p\in P$.  For each pair $p\preceq q$, let the morphism $\Phi(p,q)\colon C_*(p)\to C_*(q)$ be the induced chain map between subdivisions.  Because of the grading on $C_*(p)$, we denote by $C^{k,i}(X)=C^k(X;\calC_i)$ the $k$-cochains on $X=\order(P)$ with values in the $i$-chains of triangulations.  Hence, for a cochain $\varphi\in C^{k,i}$ and a $k$-simplex $p_o\cdots p_k$ of the order complex $X$, we have $\varphi(p_0\cdots p_k)\in C_i(p_k)$.  This gives us a double complex
\begin{equation}
\calC^{*,*} = \{C^{*,*},\d,\partial\}
\end{equation}
where $\d$ is the coboundary operator for order cohomology and $\partial$ the simplicial boundary opertaor.  Similarly, let cycle, boundary, and homology-valued systems be denoted by $\calZ$,$\calB$, and $\calH$.  We adopt the notation
\begin{equation}
Z^{k,i}(X) = C^k(X;\calZ_i)\quad\quad B^{k,i} = C^k(X;\calB_i)\quad\quad H^{k,i} = C^k(X;\calH_i)
\end{equation}
and we have associated double complexes $\calZ^{*,*}$, $\calB^{*,*}$, and $\calH^{*,*}$.

From various short-exact sequences between these systems, we may extract relationships between their order cohomologies.  Indeed, from the short-exact sequences
\begin{gather}
0\to Z^{*,i}\to C^{*,i}\stackrel{\partial}{\to} B^{*,i-1}\to 0\\
0\to B^{*,i}\to Z^{*,i}\stackrel{[\medspace]}{\to} H^{*,i}\to 0
\end{gather}
we obtain long-exact sequences
\begin{align}
\cdots\to H^k(X;\calZ^i)\to H^k(X;\calC^i)\to H^k(X;\calB^{i-1})\to H^{k+1}(X;\calZ^i)\to\cdots\\
\cdots\to H^k(X;\calB^i)\to H^k(X;\calZ^i)\to H^k(X;\calH^i)\to H^{k+1}(X;\calB^i)\to\cdots
\end{align}
respectively.

Now, fix a homology class in $H_i(M)$.  This defines an element $\eta\in H^{0,i}$ which by a choice of a cycle-representative for each triangulation defines a cycle-valued cochain $\zeta\in Z^{0,i}$.  For every $1$-simplex $pq$ in $X=\order(P)$, the difference
\begin{equation}
(\d\zeta)(pq) = \zeta(q)-\Phi(p,q)\zeta(p)
\end{equation}
of cycle-representatives between subdivisions is a boundary, and hence $\partial$-lifts to a $(i+1)$-chain in $C_{i+1}(q)$.  Thus, there is a $\partial$-lift of $\d\zeta$ to an element $\gamma\in C^{1,i+1}$ satisfying $\partial\gamma=\d\zeta$.  We will call $\gamma$ a \emph{transgression relative $\zeta$}.  When $\d\gamma=0$, we say that the transgression $\gamma$ is \emph{path-independent}. The name is motivated by the property that for any three triangulations $p\preceq q\preceq r$, a path-independent transgression satisfies the suggestive relation
\begin{equation}
\gamma(qr)+\Phi(p,q)\gamma(pq) = \gamma(pr)
\end{equation}

We can now state one of our main results:
\begin{thm}\label{path-ind}
Suppose the system $\calC$ has no non-trivial $i$-cycle assignment, or equivalently that the kernel of $\d_Z^0\colon Z^{0,i}\to Z^{1,i}$ is trivial.  Then, for any cycle-representative assignment $\zeta\in Z^{0,i}$, there is a path-independent transgression precisely when the cycles assigned are actually boundaries, $\partial\zeta=0\in H^{0,i}$.
\end{thm}
\begin{proof}
Let $x\in H^{0,i}$ be the fixed homology class in $H_i(M)$ and $a\in Z^{0,i}$ any choice of cycle representatives.  The theorem states that under the given kernel hypothesis, $x$ is trivial iff there exists a $0\neq \varphi\in C^{1,i+1}$ such that $\d\varphi=0$ and $\partial\varphi=\d_Z a$, that is, $\varphi$ is a path-independent transgression for $x$ relative $a$.

We have the following diagram:
\begin{equation}
\begin{CD}
x\in H^{0,i} @<[\quad]<< Z^{0,i}\ni a\\
& & @VV\d_Z V\\
& & B^{1,i} @<\partial << C^{1,i+1}\ni \varphi
\end{CD}
\end{equation}

Suppose $x\in H^{0,i}$ is trivial.  If the cycle-representatives are all trivial, then certainly the zero transgression works.  Thus, suppose $a$ is any \emph{non-trivial} choice of cycle-representatives for $X$.  From the commutativity of the diagram:
\begin{equation}
\begin{CD}
H^{0,i} @<[\quad]<< Z^{0,i} @<\partial<< C^{0,i+1}\\
& & @VV\d_Z V  @VV\d V\\
& & B^{1,i} @<\partial << C^{1,i+1}\ni \varphi
\end{CD}
\end{equation}
we may lift $a\in Z^{0,i}$ to $C^{0,i+1}$ and push-down via $\d$ to obtain a path-independent transgression relative cycle-representatives $a$.  Note that this transgression is not trivial because we are assuming that the system has no non-trivial cycle assignment, $\ker \d_Z^0=0$.  Thus, if the boundary assignment is non-trivial, then one can find a non-trivial path-independent transgression relative to it.

Conversely, suppose that a path-independent transgression exists.  We will show that $x=0$.  First, note that $[\varphi]\in H^{1,i+1}$ is a $\d$-cohomology class.  From the short-exact sequence
\begin{equation*}
0\to Z^{*,i+1}\stackrel{j}{\to} C^{*,i+1}\stackrel{\partial}{\to} B^{*,i}\to 0
\end{equation*}
we have the long-exact sequence
\begin{equation*}
\cdots\to H^1(M;\calZ^{i+1})\stackrel{j_*}{\to} H^1(M;\calC^{i+1})\stackrel{\partial_*}{\to} H^1(M;\calB^i)\stackrel{d^1}{\to} H^2(M;\calZ^{i+1})\to \cdots
\end{equation*}
Since $\partial\varphi=\d_Z a$, we have $\partial_*[\varphi]=0$ so that $[\varphi]=j_*[\psi]$ for some $\d$-closed $\psi\in H^1(M;\calZ^{i+1})$.  Writing $\varphi = j_*\psi+\d\tau$ for some $\tau\in C^{0,i+1}$, we have $\d_Z a=\partial\varphi=\partial\d\tau$.  Since transgressions are determined up to cycle-assignments, we may modify $\varphi$ to $\varphi-j_*\psi$ and assume that our transgression actually $\d$-lifts, $\varphi=\d\tau$.  By commutativity of the diagram
\begin{equation}
\begin{CD}
x\in H^{0,i} @<[\quad]<< Z^{0,i} @<\partial<< C^{0,i+1}\ni \tau\\
& & @VV\d_Z V  @VV\d V\\
& & B^{1,i} @<\partial << C^{1,i+1}
\end{CD}
\end{equation}
we have $\partial\tau-a\in \ker \d_Z^0$.  But $\ker \d_Z^0$ is trivial so that $\partial\tau=a$ and $x=[a]=0\in H^{0,i}$ is trivial.
\end{proof}

An example of a system where the hypothesis on $\d_Z^0$ fails is the poset of triangulations of a fixed simplicial complex.  The kernel is non-trivial because any cycle of the fixed complex includes into every subdivision of the complex.  This example is addressed in Corollary \ref{fixed}.  However, if the underlying space is not already a simplicial complex, the hypothesis is mild since homeomorphisms from simplicial complexes to a space admitting triangulations do not generally ``fix cycles'' in the image.  More precisely, given a space $M$ admitting triangulations, and choosing arbitrary triangulations $\varphi,\psi\colon |K|,|L|\to M$, the homeomorphism $\psi^{-1}\circ \varphi:|K|\to |L|$ generally does not send realizations of simplicial.cycles in $K$ to realizations of simplicial cycles in $L$.

\begin{ex}[Euler transgressive chains]
An example of a combinatorial transgression is that associated to \emph{Euler cycles}.  Let $X$ be a simplicial complex, and define the \emph{Euler cycle} \cite{forman} for $X$ to be the $0$-chain
\begin{equation}\label{eulerchain}
e(X) = \sum_{v\in X} e(v,X)v
\end{equation}
over vertices $v$, where
\begin{equation}
e(v,X) = \sum_{k=0} (-1)^k \frac{\text{\# $k$-simplices containing $v$}}{k+1}
\end{equation}
The coefficients of the $0$-chain sum to the Euler characteristic because for each $k$-simplex, unity is being equally divided among its $(k+1)$-vertices, and we are taking an alternating sum over dimension.  When $X$ is actually a combinatorial manifold, rational combinatorial characteristic classes exist, and the Euler cycle is a representative cycle for the Poincar\'{e}-dual of the rational Euler class.  This explains the nomenclature.

\begin{remark}
We may rewrite the preceding coefficient expression as
\begin{equation}
e(\lk(v,X)) = 1+ \sum_{k=0} (-1)^{k+1} \frac{\text{\# $k$-simplices in $\lk(v,X)$}}{k+2}
\end{equation}
Using Gaifullin's terminology\cite{gaifullin}, $e(\lk)$ is called a \emph{local formula} and $e(X)$ a \emph{characteristic local cycle}.
\end{remark}

Since the Euler characteristic is invariant under subdivision, if $i\colon Y \to X$ is a subdivision of $Y$, the difference $e(X)-i_*e(Y)$ is a $0$-chain whose coefficients sum to zero, and hence, the boundary of a $1$-chain on $X$ which is our desired transgression.  The preceding theorem implies that for spaces satisfying the mild hypothesis, transgressions relative to Euler cycles are not path-independent except when the Euler characteristic of the space vanishes.  In Section \ref{local} we will obtain a transgression relative Euler cycles which is locally determined in a precise manner.
\end{ex}

\begin{ex}
Let $M$ be a combinatorial manifold.  Then, rational combinatorial Pontryagin classes exist.  Fixing a Pontryagin cohomology class, suppose representative cycles are chosen for the homology duals.  This defines a cycle-assignment on the poset of triangulations, and if $M$ satisfies the hypothesis of the above theorem, transgressions relative to such combinatorial Pontryagin cycle assignments are not path-independent except when the Pontryagin class is trivial.
\end{ex}

\begin{ex}
Suppose for each triangulation the chain spaces are endowed with an inner product.  Then, Hodge theory provides unique cycle representatives for each homology class.  For spaces satisfying the hypothesis, transgressions relative to these harmonic cycles are not path-independent whenever a non-zero harmonic cycle is assigned.
\end{ex}

Suppose a path-independent transgression of $k$-cycles exists on $\calC$ a system of chain complexes indexed by poset $P$.  How can it be modified?  To preserve path-independence, we can only modify by adding path-independent cochains, that is, 1-cocycles.  However, to preserve the transgression property, these additional cocycles should take values in $(k+1)$-cycles.  Hence, the set of path-independent transgressions is the kernel of the map $\d_Z^1\colon Z^{1,k+1}\to Z^{2,k+1}$.

Now we consider the special case of the subdivision poset for a fixed simplicial complex:

\begin{cor}\label{fixed}
For a fixed triangulation, the poset of its subdivisions admits a path-independent transgression (possibly trivial).  Using the canonical inner product on chain spaces of the subdivisions yields an unique path-independent transgression.
\end{cor}
\begin{proof}
We use the notation in Theorem \ref{path-ind}.  First, note that by Proposition \ref{wedge}, the order cohomology of the poset of subdivisions for a fixed triangulation is acyclic.  The proof of Theorem \ref{path-ind} shows that if a path-independent transgression $\varphi$ exists, then it is actually a $\d$-coboundary $\varphi=\d\tau$ and satisfies $\partial\tau-a\in \ker \d_Z^0$ so that this is a sufficient condition.  Now, any kernel element may be realized by choosing a representative $i$-cycle on the base triangulation and extending over the poset to obtain a $0$-cochain $\eta\in \ker \d_Z^0$.  In fact, for any such choice, the cochain $a-\eta$ then $\partial$-lifts, and we can push-down to obtain a path-independent transgression, and every path-independent transgression arises in this manner.  Finally, if we use the canonical inner product, there is a distinguished cycle representative on the base triangulation as well as distinguished $\partial$-lifts.  Thus, we obtain an unique path-independent transgression.
\end{proof}

We can explicitly write the set of path-independent transgression as follows: let $z$ be a $k$-cycle on base (that is, minimal) triangulation $p\in P$ homologous to $a_p$.  Writing $\Phi_*z$ for the induced element of $C^0(X;\calZ_k)$, the cochain $a-\Phi_*z\in C^0(X;\calB_k)$ admits a $\partial$-lift.  Choose the standard inner product on the chain spaces $C_{k+1}(q)$ for $q\in P$.  If we denote by $h_q$ the unique norm-minimizing element of the subspace $\partial_{k+1}^{-1}(a_q-\Phi(p,q)z)$, then the set of $\partial$-lifts is $h+C^0(X;\calZ_{k+1})$.  Finally, the set of path-independent transgressions is
\begin{equation*}
\d^0(h+C^0(X;\calZ_{k+1})) = \d^0 h + \d^0_Z(\calZ^{0,k+1}) = \d^0 h + \ker \d^0_Z
\end{equation*}
where $\ker \d^0_Z$ in the last inequality refers to a subsystem of $\calZ^{1,k+1}$ and follows by acyclicity of posets with unique minimal element.

\begin{remark}
Note that the above arguments are homological, and thus statements about certain systems of chain complexes, namely those whose induced homology maps $\Phi(p,q)\colon H^*(C_*(p))\to H^*(C_*(q))$ are isomorphisms for $p\preceq q$.  In particular, the path-independent transgression results are true for other combinatorial functors from topological spaces to chain complexes, e.g. cubical complexes, for which subdivision is well-defined.
\end{remark}

Finally, the results in this section are also valid on the poset of simplicial isomorphism classes of triangulations if isomorphisms are chosen.  Let $Q=\overline{P}$ be the poset of triangulations modulo simplicial isomorphism, and define a system $\overline{\calC}$ by assigning vector spaces isomorphisms $C_*(T_q)\to V_q$ where $T_q$ is a representative triangulation for $q\in Q$.  Morphisms $\Phi$ of the system $\calC$ define morphisms $\overline{\Phi}$ for $\overline{\calC}$ by stipulating commutativity of the diagram:
\begin{equation}\label{isoclass}
\begin{CD}
C_*(T_q) @>\Phi >> C_*(T_{q'})\\
@V\cong VV @VV\cong V\\
V_q @>\overline{\Phi}>> V_{q'}
\end{CD}
\end{equation}
for $q,q'\in Q$, and one readily checks that the map $\overline{\Phi}$ is well-defined.  Then, we may form an order chain complex in this context and compute order cohomology using representatives for the isomorphism classes and the vertical isomorphisms in Diagram (\ref{isoclass}).  Hence, the results of this section are valid on the poset of triangulations modulo simplicial isomorphism.

\begin{ex}
Consider the poset of triangulations of $S^1$ directed by subdivision.  There are many minimal triangulations, but they are all simplicially isomorphic.  Hence, the associated poset of simplicial isomorphism classes has an unique minimal element and so admits a path-independent transgression.
\end{ex}

\section{Local Formula for a Transgression Relative Euler Cycles}\label{local}

The Chern-Weil theory of characteristic classes in the smooth category provides not only local representatives for characteristic classes, but also a canonical local formula for transgressions relative the difference of such representatives (see Appendix for details).  In this section we exhibit a local formula for a transgression relative Euler cycles.  The notion of locality is that of a \emph{party} (see Definition \ref{party}).

Recall from Example \ref{eulerchain} above that for a simplicial complex $X$, the \emph{Euler cycle} is the $0$-chain
\begin{equation*}
e(X) = \sum_{v\in X} e(v,X)v
\end{equation*}
over vertices $v$, where
\begin{equation*}
e(v,X) = \sum_{k=0} (-1)^k \frac{\text{\# $k$-simplices containing $v$}}{k+1}
\end{equation*}
Since for each $k$-simplex, unity is being equally divided among its $(k+1)$-vertices, we see that for any simplicial complexes $X,Y$ we have $e(X\cup Y)=e(X)+e(Y)-e(X\cap Y)$ which we will denote compactly by $e(X\cup Y)=\IE\{e(X),e(Y)\}$.  More generally, for a collection of simplicial complexes $X_i$, the Euler cycle for the union $X=\cup X_i$ is given by inclusion-exclusion
\begin{equation}
e(X) = e(\cup X_i) = \IE\{e(X_i)\} 
\end{equation}
where $\IE\{e(X_i)\}$ denotes inclusion-exclusion relative the collection $\{X_i\}$.

In fact, since any simplicial complex is the union of its simplices, we see that the Euler cycle is characterized uniquely by the following properties:
\begin{itemize}
\item (Inclusion-Exclusion) $e(X\cup Y)=e(X)+e(Y)-e(X\cap Y)$ for any simplicial complexes $X,Y$
\item (Calibration) $e(\s^{(k)})= (\sum_v v)/(k+1)$ for any $k$-simplex $\s^{(k)}$, where the sum is over vertices
\item (Wellness) $e(\emptyset)=0$
\end{itemize}
Note that if we write $\int \a$ for the sum of the coefficients of any $0$-chain $\a$, then $\int e(Y)= \chi(Y)$ the Euler characteristic of $Y$.

To state our local formula result, we need a notion of locality:
\begin{defn}\label{party}
Let $i\colon N\to M$ be a subdivision of simplicial complexes.  A \emph{$k$-party} of $M$ is a union of $k$-simplices of $M$ which coincides with the image of a $k$-simplex of $N$ under the simplex map $i$.
\end{defn}

\begin{thm}
There is a canonical formula (with rational coefficients) for the transgression relative the Euler cycles.  This formula is local in the sense that the coefficient for any edge in the transgression chain is determined by the union of parties containing that edge.
\end{thm}
\begin{proof}
Let $i\colon Y\to X$ be a subdivision.  Write $Y=\cup P_i$ where $P_i$ are the maximal faces of $Y$ (that is, the maximal elements of the face poset of $Y$), and let $\{\calN_k^j\}$ denote the faces that are $(j+1)$-fold intersections of elements from the collection $\{P_i\}$.  For example, $\{\calN_k^0\}=\{P_i\}$.  Then,
\begin{equation*}
e(Y)= \IE\{e(P_i)\} = \sum_j (-1)^j \sum_k e(\calN_k^j)
\end{equation*}
Now, $i_*\calN_k^j$ is a party of $X$, say denoted $Q_k^j$.  Then,
\begin{equation*}
e(X) = e(\cup i_*P_i) = \sum_j (-1)^j \sum_k e(Q_k^j)
\end{equation*}
so that
\begin{equation}\label{difference}
e(X)-i_*e(Y) = \sum_{j,k} (-1)^j \left( e(Q_k^j) - i_*e(\calN_k^j) \right)
\end{equation}
But $\int e(Q_k^j)=1=\int i_*e(\calN_k^j)$, so that the difference terms within parentheses in Equation (\ref{difference}) are boundaries.  Hence, we can $\partial$-lift locally.  This can be effected by a choice of spanning tree for the vertex set of the party $Q_k^j$.  Averaging over all such choices yields a canonical choice.  The coefficient of any edge in $X$ is thus determined by the union parties containing that edge.
\end{proof}

\begin{remark}
If $X$ is a pure simplicial complex (for example, a combinatorial manifold), then $j$ in the preceding proof may be interpreted as codimension.  Also, note that the proof only required that (1) the Euler cycles $e$ satisfy inclusion-exclusion, and (2) $e$ assigns homologous chains to a simplex and any subdivision of it.
\end{remark}

\begin{remark}
If we use the canonical inner product on the chain spaces, then we obtain unique $\partial$-lifts so that there is a canonical locally-determined transgression without averaging.
\end{remark}

\begin{remark}
The difference of two arbitrary representatives of a given homology class may not possess $\partial$-lifts in any local sense.  Consider the real line $\R$ triangulated with vertex set $\Z$.  Let $v_i$ denote the vertex at integer $i$.  The $0$-chain $v_i-v_{-i}$ \emph{cannot} in general possess a local $\partial$-lift because the integer $i$ may be chosen arbitrarily large so that any local $\partial$-lift must assign some of the intervening edges a coefficient of zero.
\end{remark}

\begin{remark}
An analogous result holds true for Euler cycles relative cubical complexes by amending the definition for Euler cycles to divide unity among the vertices of cubes of various dimensions.
\end{remark}

\section{Relations between Order Cohomology and Sheaf Theory}\label{sheaf}

We conclude with a sheaf-theoretic interpretation of order cohomology over sufficiently nice posets such as lattices.  (For the use of sheaves over (finite) posets, see Baclawski\cite{baclawski}.)  We will show that for a fixed poset, each system has an associated sheaf relative the upper-ideal topology and vice-versa.  However, these associations are not inverses due to the carseness of the upper-ideal topology (see Example \ref{notinverse}).  In the case of lattices, order cohomology coincides with sheaf cohomology with respect to the system's associated sheaf (see Proposition \ref{coincide}). 

Let $P$ be a poset and $\calG$ a system of abelian groups ordered by $P$.  The upper ideals of $P$ generate a topology on $P$ and hence also on $X=\order(P)$ the order complex.  We define a presheaf on $X$ as follows: Associate to each open set $U$, the cochain group $C^*(U;\calG)$ which we will denote by $\calG(U)$, and for $V\subset U$ open sets, define the restriction morphism on sections $\calG(U)\to \calG(V)$ to be just restriction of cochains.  Next we check that it satisfies the axioms for a sheaf: We will write $\uparrow x=\{y\in P| x\preceq y\}$ for the upper-ideal generated by $x\in P$.
\begin{itemize}
\item \emph{Monopresheaf axiom}: Note that for any open set $V$ in the upper-ideal topology, if $x\in V$, then $\uparrow x\subset V$.  Let $U$ be an open set with cover $\calU$.  By the preceding remark, every simplex subordinate to $U$ is subordinate to some cover element of $\calU$.  Thus, any cochain vanishing on all cover elements vanishes globally.
\item \emph{Gluing axiom}: Let $U$ be an open set with cover $\calU$.  Suppose we have local sections $\{s_V\}_{V\in \calU}$ compatible on overlap.  Then, we may define a global section on $U$ using extension by zero on simplices not subordinate to the cover $\calU$.  Note that if we cover with basis elements, then every simplex is subordinate to some cover element, and extension is unnecessary.
\end{itemize}
Thus, $\calG$ is a sheaf of abelian groups.  Since for any $z\in P$, the smallest open set containing $z$ is the upper ideal generated by $z$, the stalk over $z$ is precisely the group $\calG(\uparrow z)$.  Conversely, given a sheaf $\calF$ of abelian groups, we have a naturally associated system $\calG$ of abelian groups given by $\calG_z=\calF_z=\calF(\uparrow z)$ with morphisms $\calG_x\to \calG_y$ being the restriction map on sections of $\calF$, namely $\calF(\uparrow x)\to \calF(\uparrow y)$.

We summarize the preceding associations: Given a system of abelian groups $\calG=\{\calG_x,\Phi(x,y)\}$ directed by poset $P$, the associated sheaf is defined by $\calF(U)=C^*(U;\calG)$ for any open set $U$ under the upper ideal topology on $P$, with restriction maps given by restriction of $\calG$-valued cochains.  Conversely, given a sheaf $\calF$ of abelian groups over poset $P$ under the upper ideal topology, the associated system is given by $\calG_x=\calF(\uparrow x)$ with morphisms $\Phi(x,y)\colon \calG_x\to \calG_y$ given by restriction for $x\preceq y$.

These associations are not inverses of each other as evidenced by the following example:
\begin{ex}\label{notinverse}
Let $P=\{a,b\preceq c\}$ a poset of three elements.  Suppose $\calG$ is a system of abelian groups directed by $P$.  The associated sheaf is given by
\begin{align*}
\calF(\uparrow x) &= C^*(\uparrow x;\calG)\cong (\calG_x\oplus\calG_c)\oplus \calG_c\quad\quad x=a,b\\
\calF(\uparrow c) &= C^*(\uparrow c;\calG)\cong (\calG_c)\oplus 0
\end{align*}
with restriction maps $\calF(\uparrow a), \calF(\uparrow b)\to \calF(\uparrow c)$ both given by $\pi_2\oplus 0$ where $\pi_2$ is restriction onto the second summand of $C^0$.  Now, the system $\calG'$ associated with sheaf $\calF$ is given by
\begin{align*}
\calG'_x &=\calF(\uparrow x)\quad x\in P\\
\Phi'(x,c) &=\pi_2\oplus 0\quad x=a,b
\end{align*}
In particular, note that $\calG_x\hookrightarrow \calG'_x\cong (\calG_x\oplus\calG_c)\oplus\calG_c$ for $x=a,b$.

Conversely, suppose we start with a sheaf of abelian groups $\calF$ over $P$ under the upper-ideal topology.  The associated system $\calG$ is given by $\calG_x=\calF_x=\calF(\uparrow x)$ for $x\in P$ and $\Phi(x,y)$ is the restriction map $\calF(\uparrow x)\to \calF(\uparrow y)$ for $x\preceq y$.  Now, the sheaf $\calF'$ associated with this system $\calG$ is defined by
\begin{align*}
\calF'(\uparrow x) &= C^*(\uparrow x;\calG)=(\calF_x\oplus\calF_c)\oplus \calF_c\quad\quad x=a,b\\
\calF'(\uparrow c) &= C^*(\uparrow c;\calG)=\calF_c
\end{align*}
with restriction maps $\calF'(\uparrow x)\to \calF(\uparrow c)$ for $x=a,b$ also given by $\pi_2\oplus 0$.

The failure in both cases is due to the coarseness of the ideal toplogy.  Passage from system to sheaf packages data on an entire upper ideal so that the individual groups $\calG_x$ are no longer directly accessible.
\end{ex}

Let $\calF$ be a sheaf over topological space $X$.  Recall that $\calF$ is \emph{flasque} provided for any open set $U$ the restriction morphism is surjective, and if $\calF$ is flasque and $\calU$ an open cover of $X$, the associated C\v{e}ch cohomology $\tilde{H}^*(\calU;\calF)$ with sheaf coefficients is actually acyclic.  In our case, fixing any open cover (relative the upper-ideal topology) of $P$ and using the zero extension over simplices not subordinate to the covering, we see that for any direct system, the associated sheaf is flasque.  Thus, if we form a C\v{e}ch-order double complex analogous to the C\v{e}ch-deRham complex, then the order cohomology always coincides with the diagonal cohomology of the double complex.  The existence of a good cover relative order cohomolgy would then yield the equality of order and sheaf cohomologies.
\begin{prop}\label{coincide}
Let $\calG$ be a system of abelian groups directed by $P$ and $\calF$ its associated sheaf.  Suppose unique upper bounds exist in $P$.  Then, order and sheaf cohomologies coincide, $H(X;\calG)\cong \tilde{H}^*(X;\calF)$.  In particular, this is true for lattices.
\end{prop}
\begin{proof}
Consider the cover $\calU$ of $X$ by basis elements $\{\uparrow z|z\in P\}$.  Since unique upper bounds exist, any intersection is principally generated, hence acyclic (see Proposition \ref{acyclic}).  Thus, $\calU$ is a good cover, and the order cohomology coincides with the C\v{e}ch cohomology relative $\calU$.  However, $\calU$ is the finest refinement over $X$ so that order cohomology indeed coincides with sheaf cohomology.
\end{proof}

\begin{remark}
Good covers relative order cohomology need not exist for a general direct system.  Let $P=\{a,b\preceq c,d\preceq e\}$ be the underlying poset.  Any proper open set containing either $a$ or $b$ but not both will contain the subposet $c,d\preceq e$.  Suppose for example, that the open set contains $a$ but not $b$.  Setting $V_a=V_c=V_d$ and $V_e\neq 0$, we see as in Example \ref{wedge} that there is no good cover relative order cohomology.
\end{remark}

The utility of this sheaf-theoretic formulation of order cohomology is still unclear to the author.

\section{Appendix: Differential-Geometric Motivation}\label{diffgeo}

This appendix provides a brief overview of the Chern-Weil theory of characteristic classes that motivated the questions on path-independence and local formulae for transgressions.  In particular, we exhibit the local canonical transgression associated to two connections, and we use the Cartan Homotopy Formula to show that another natural transgression (namely, the difference of Chern-Simons forms) is path-independent.

Let $M$ be a smooth manifold and $\xi=(\pi,E,M)$ be a smooth $n$-bundle.  Let $\nabla$ be a connection on $\xi$ and $R_{\nabla}$ the associated curvature operator, which locally can be considered as an $\End(\xi)$-valued $2$-form $\Om_{\nabla}$.  Let $I^*(GL_n)$ be the set of invariant polynomials on the general linear group $GL_n$.  Thus, elements of $I^*(GL_n)$ are polynomial functions on the entries of $n\times n$-square matrices satisfying $P(X)=P(AXA^{-1})$ for all invertible $A$.

The central result of Chern-Weil theory is that given a degree $k$ invariant polynomial $P\in I^k(GL_n)$, the local form $P(\Om_{\nabla})\in \calA^{2k}(M)$ is closed, and the associated cohomology class is a characteristic class independent of the connection chosen.  Because the curvature forms transform via conjugation under change of frame field, the invariant polynomial guarantees that the \emph{local expressions} $P(\Om_{\nabla})$ represent the same form on the overlap of trivializations, and hence define a global form.

If $\xi$ is an oriented, real vector bundle of even rank, say $n=2k$, then endowing $\xi$ with a Riemannian metric, the structure group reduces to $SO(n)$, whence the Pfaffian becomes an invariant polynomial.  Choosing a metric-compatible connection, the curvature forms $\Om_{\nabla}$ are skew-symmetric, and the cohomology class $\Pf(\Om/2\pi) \in H^n(M)$ corresponds to the integral \emph{Euler class} of the bundle.

One approach to showing that the characteristic classes of Chern-Weil theory are independent of the choice of connection is to pick two connection forms $\om_0,\om_1$ and consider a $1$-parameter family of connections $\om_t$ between them with associated deformation $\Om_t$ of curvature forms.  One can show that for any invariant polynomial $P$, the form
\begin{equation}
P(\Om_1)-P(\Om_0) = \int_0^1 \frac{d}{dt}P(\Om_t) dt = d\varphi(\om_t)
\end{equation}
is exact, where the form
\begin{equation}\label{transgression}
\varphi(\om_t)= k \int_0^1 P(\om_t'\wedge \Om_t^{k-1})dt
\end{equation}
is called a \emph{transgression} relative the connections representated by $\om_0,\om_1$.  Because $\om_t'$ also transforms by conjugation, Equation (\ref{transgression}) actually defines a form globally.  Though transgressive forms are determined up to closed forms, because the space of connections is an affine space, there is a natural ``straight-line'' deformation between any two connections, and Equation (\ref{transgression}) actually gives a local formula for a canonical choice of transgression.

For the Euler class, we consider an analogous deformation of Riemannian metrics, and hence associated Levi-Civita connections.  Again, we may obtain a local formula for a canonical choice of transgression, now between Euler forms.

We now use the \emph{Cartan Homotopy Formula} to show that a path-independent transgression is always possible. Consider a polynomial algebra over variables $x,y$ of degrees $1,2$ respectively, and suppose both variables are parameterized by another variable $t$.  We introduce anti-derivations $d,l_t$ subject to the following:
\begin{alignat}{2}
dx &= y\pm x^2  \qquad l_t x &= 0 \\
dy &= \pm[x,y] \qquad l_t y &=\partial x/\partial t
\end{alignat}
The variables $x,y$ mimic the connection and curvature forms, respectively, and the $d$-relations above thus express Cartan's second structural equation and the Bianchi Identity.  The choice of sign depends on whether the structure group acts on the left or right.  One checks that
\begin{equation}\label{babyhomotopy}
l_t d + dl_t = \frac{\partial}{\partial t}
\end{equation}
This construction extracts the algebra involved in applying the Chern-Weil theory to a one-parameter deformation of the connection.  For any polynomial $f$ in the algebra, we define
\begin{equation}
Hf = \int_0^1 l_t f \medspace dt
\end{equation}
and write $\Delta f(x_t,y_t)=f(x_1,y_1)-f(x_0,y_0)$ for the (directed) difference of the endpoints of the deformation.  Then, by integrating Equation (\ref{babyhomotopy}), we see that $H$ defines a homotopy operator via the \emph{Cartan Homotopy Formula}:
\begin{equation}
\Delta f= \int_0^1 \frac{\partial f}{\partial t}\medspace dt = (Hd +dH)f
\end{equation}
If $P$ is an invariant polynomial of degree $k$, then
\begin{equation*}
l_tP(y^k) = kP(l_t y\cdot y^{k-1}) = kP\left(\frac{\partial x}{\partial t} y^{k-1}\right)
\end{equation*}
so that
\begin{align*}
P(y_1^k)-P(y_0^k) &= \Delta P(y_t^k) = (dH+Hd)P(y_t^k) = dHP(y_t^k)\\
&= d\int_0^1 k P\left(\frac{\partial x}{\partial t} y^{k-1}\right)\medspace dt
\end{align*}
which recovers the transgression in Equation (\ref{transgression}).

For any connection, define the associated \emph{Chern-Simons form}\cite{chern2} as the transgression using Equation (\ref{transgression}) between $0$ and the given connection:
\begin{equation*}
TP(\om) = \varphi(0\om) = k \int_0^1 P(\om \wedge \Phi_t^{k-1})\medspace dt
\end{equation*}
where $\Phi_t = t\Om + (1/2)t(t-1)[\om,\om]$.  It satisfies
\begin{equation}
dTP(\om) = P(\Om)
\end{equation}
Now consider a deformation of Chern-Simons forms:
\begin{equation}
TP(\om_t) = \varphi(0\om_t) = k \int_0^1 P(\om_t \wedge \Phi_{s,t}^{k-1})\medspace ds
\end{equation}
where now $\Phi_{s,t} = s\Om_t + (1/2)s(s-1)[\om_t,\om_t]$.  Then,
\begin{equation}
dTP(\om_t) = P(\Om_t)
\end{equation}
and by the homotopy formula,
\begin{equation}
\Delta TP(\om_t) = (Hd+dH)TP(\om_t) = \varphi(\om_t) + dH(TP(\om_t))
\end{equation}
Note that $\Delta TP(\om_t)$ also defines a transgression in addition to $\varphi(\om_t)$. Finally, writing $Q(\om_a\om_b)=\Delta TP(\om_t)$ for the difference of Chern-Simons associated to the deformation $\om_t$ from $\om_a$ to $\om_b$, we find that 
\begin{equation*}
dQ(\om_a\om_b) = d\Delta TP(\om_t) = d(TP(\om_b)-TP(\om_a)) = P(\Om_b)-P(\Om_a) = (\d P)(\om_a\om_b)
\end{equation*}
and
\begin{align*}
(\d Q)(\om_a\om_b\om_c) &= Q(\om_a\om_b)+Q(\om_b\om_c)-Q(\om_a\om_c)\\
&= (TP(\om_b)-TP(\om_a))+(TP(\om_c)-TP(\om_b))-(TP(\om_c)-TP(\om_a))\\
&= 0
\end{align*}
where $\d$ is the C\v{e}ch-like coboundary operator in order cohomology.  Hence, $Q=\Delta TP$ the difference of Chern-Simons forms defines a path-independent transgression.

\bibliographystyle{plain}
\bibliography{ordercohomology}

\end{document}